\documentclass[11pt]{amsart}

\usepackage{array,float,hyperref,yhmath}
\usepackage{amssymb, amsmath, amsthm, amsbsy, amscd, mathrsfs,stmaryrd}
\usepackage[normalem]{ulem}
\usepackage{hyperref}

\numberwithin{equation}{section}
\numberwithin{figure}{section}

\theoremstyle{plain}
 \newtheorem{theorem}{Theorem}[section]
 \newtheorem{proposition}[theorem]{Proposition}
 \newtheorem{lemma}[theorem]{Lemma}
 \newtheorem{corollary}[theorem]{Corollary}

\theoremstyle{definition}

\theoremstyle{remark}

 \newtheorem{remark}[theorem]{Remark}
 \newtheorem*{acknowledgements}{Acknowledgements}
 \newtheorem*{funding}{Funding}
 
 \newtheorem{example}[theorem]{Example}

\newcommand{\nl}{\triangleleft}
\newcommand{\nlgr}{\triangleleft_{\textup{gr}}}

\newcommand{\be}{\begin{enumerate}}
\newcommand{\ee}{\end{enumerate}}

\newcommand{\brem}{\begin{remark}}
\newcommand{\erem}{\end{remark}}

\newcommand{\bp}{\begin{proof}}
\newcommand{\ep}{\end{proof}}

\newcommand{\Fq}{\ensuremath{\mathbb{F}_q}}

\newcommand{\N}{\ensuremath{\mathbb{N}}}
\newcommand{\Zp}{\ensuremath{\mathbb{Z}_p}}

\newcommand{\C}{\ensuremath{\mathbb{C}}}

\newcommand{\Q}{\ensuremath{\mathbb{Q}}}
\newcommand{\Z}{\ensuremath{\mathbb{Z}}}

\newcommand{\bfz}{{\bf 0}}

\newcommand{\bfa}{\ensuremath{\mathbf{a}}}

\newcommand{\bfe}{\ensuremath{\mathbf{e}}}
\newcommand{\bfE}{\ensuremath{\mathbf{E}}}
\newcommand{\bff}{\ensuremath{\mathbf{f}}}
\newcommand{\bfF}{\ensuremath{\mathbf{F}}}
\newcommand{\bfg}{\ensuremath{\mathbf{g}}}

\newcommand{\bfr}{\ensuremath{\mathbf{r}}}

\newcommand{\bfS}{\ensuremath{\mathbf{S}}}

\newcommand{\bfy}{\ensuremath{\mathbf{y}}}

\newcommand{\bfB}{\ensuremath{\boldsymbol{B}}}
\newcommand{\bfG}{\ensuremath{\mathbf{G}}}
\newcommand{\bfH}{\ensuremath{\mathbf{H}}}

\newcommand{\bfX}{\ensuremath{\mathbf{X}}}
\newcommand{\bfY}{\ensuremath{\mathbf{Y}}}

\newcommand{\mcB}{\mathcal{B}}

\newcommand{\mcO}{\mathcal{O}}

\newcommand{\mcN}{\ensuremath{\mathcal{N}}}
\newcommand{\mcV}{\ensuremath{\mathcal{V}}}

\newcommand{\mfh}{\ensuremath{\mathfrak{h}}}

\newcommand{\rk}{\ensuremath{{\rm rk}}}

\newcommand{\wt}{\ensuremath{\widetilde}}

\newcommand{\wtB}{\ensuremath{\widetilde{B}}}

\newcommand{\lri}{\mathfrak o}

\newcommand{\la}{\langle}
\newcommand{\ra}{\rangle}

\newcommand{\Gri}{\ensuremath{\mathcal{O}}}

\newcommand{\mfp}{\mathfrak{p}}

\newcommand{\rarr}{\rightarrow}

\newcommand{\udots}{\mathinner{\mskip1mu\raise1pt\vbox{\kern7pt\hbox{.}}
\mskip2mu\raise4pt\hbox{.}\mskip2mu\raise7pt\hbox{.}\mskip1mu}}

\DeclareMathOperator{\gr}{gr}
\DeclareMathOperator{\Des}{Des}
\DeclareMathOperator{\val}{val}
\DeclareMathOperator{\Spec}{Spec}
\DeclareMathOperator{\red}{red}

\DeclareMathOperator{\Id}{Id}

\DeclareMathOperator{\Mat}{Mat}

\DeclareMathOperator{\GL}{GL}

\DeclareMathOperator{\topo}{top}
\DeclareMathOperator{\trans}{tr}
\renewcommand{\epsilon}{\varepsilon}
\renewcommand{\phi}{\varphi}
\begin{document}

\title{Ideal zeta functions associated to a family of
  class-2-nilpotent Lie rings}

\date{\today} \author{Christopher Voll} \address{Fakult\"at f\"ur
  Mathematik, Universit\"at Bielefeld\\ Postfach 100131\\ D-33501
  Bielefeld\\Germany} \email{C.Voll.98@cantab.net}


%, 16W20}
% 11M41 other Dirichlet series and zeta functions
% 20E07 subgroup growth
% 11S40 Zeta functions and $L$-functions
% 16W20 Automorphisms and endomorphisms
%\makeindex 

\begin{abstract} 
  We produce explicit formulae for various ideal zeta functions
  associated to the members of an infinite family of
  class-$2$-nilpotent Lie rings, introduced in
  \cite{BermanKlopschOnn/18}, in terms of Igusa functions. As
  corollaries we obtain information about analytic properties of
  global ideal zeta functions, local functional equations,
  topological, reduced, and graded ideal zeta functions, as well as
  representation zeta functions for the unipotent group schemes
  associated to the Lie rings in question.
\end{abstract}
\maketitle

\thispagestyle{empty}

%  \setcounter{tocdepth}{4}
%  \tableofcontents

\section{Introduction and main result}
The main object of this paper is the computation of various ideal zeta
functions associated to the class-$2$-nilpotent Lie rings $L_{m,n}$
defined in \cite{BermanKlopschOnn/18}.

\subsection{Ideal zeta functions of Lie rings and
  algebras}\label{subsec:ideal}
Let $R$ be the ring of integers $\mcO$ of a number field or a compact
discrete valuation ring, such as the completion $\mcO_\mfp$ of $\mcO$
at a nonzero prime ideal $\mfp$ or a formal power series ring of the
form $\Fq\llbracket T \rrbracket$. Let $L$ be a nilpotent $R$-Lie
algebra which is free of finite rank over~$R$. The \emph{ideal zeta
  function} of $L$ is the Dirichlet generating series
$$\zeta^{\nl}_{L}(s) = \sum_{I \nl L} |L : I |^{-s},$$
enumerating $R$-ideals in $L$ of finite index in $L$. Here $s$ is a
complex variable. 

Assume now that $R=\Gri$ is the ring of integers of a number
field~$K$. For a nonzero prime ideal $\mfp \in\Spec(\Gri)$ we write
$\Gri_{\mfp}$ for the completion of $\Gri$ at $\mfp$, a complete
discrete valuation ring of characteristic zero and residue field
$\Gri/\mfp$ of cardinality $q = q_\mfp$, say. We consider
$L(\Gri_{\mfp}):=L\otimes_{\mcO}\mcO_{\mfp}$ as an $\Gri_{\mfp}$-Lie
algebra. Little more than the Chinese Reminder Theorem is necessary to
obtain the Euler product
$$\zeta^{\nl}_{L}(s) = \prod_{\mfp \in\Spec(\Gri)}
\zeta^{\nl}_{L(\Gri_\mfp)}(s);$$
cf.\ \cite[\S~3]{GSS/88}. A deep theorem, in contrast, asserts that
all the Euler factors $\zeta^{\nl}_{L(\Gri_\mfp)}(s)$ are rational
functions in $q_{\mfp}^{-s}$; cf.\
\cite[Theorem~3.5]{GSS/88}. Computing these rational functions is, in
general, a hard problem.

\subsection{The Lie rings $L_{m,n}$ and their ideal zeta functions}
In the current paper we compute explicitly, for any given $m,n\in\N$,
the ideal zeta functions of the $\lri$-Lie algebras $L_{m,n}(\lri):=
L_{m,n} \otimes_\Z \lri$, where $L_{m,n}$ is the class-$2$-nilpotent
Lie ring introduced in \cite{BermanKlopschOnn/18} and $\lri$ is a
compact discrete valuation ring of arbitrary characteristic. Our main
Theorem~\ref{thm:main} expresses the ideal zeta functions
$\zeta^{\nl}_{L_{m,n}(\lri)}(s)$ as rational functions in $q$ and
$q^{-s}$, where $q$ is the residue field cardinality of $\lri$, in
terms of Igusa functions. To recall the definition of the Lie ring
$L_{m,n}$, put
\begin{alignat*}{2}
  \bfE &= \bfE(m,n) = \{ \bfe \mid \bfe = (e_1,\dots,e_n) \in\N_0^n, & e_1 + \dots +
  e_n &= m-1\},\\ \bfF &= \bfF(m,n) = \{ \bff \mid \bff = (f_1,\dots,f_n) \in\N_0^n,
  &\; f_1 + \dots + f_n &= m\}.
\end{alignat*}
The Lie ring $L_{m,n}$ has generators
$$\{x_\bfe \mid \bfe \in \bfE \} \cup \{ y_\bff \mid \bff \in \bfF \} \cup \{ z_j \mid j\in[n]\},$$
subject to the defining relations
\begin{equation}\label{equ:def.rel}
  [x_\bfe,x_{\bfe'}] = [y_{\bff},y_{\bff'}] = [x_\bfe,z_j] =
  [y_\bff,z_j] = [z_j,z_{j'}]= 0
\end{equation}
for all $\bfe,\bfe'\in\bfE$, $\bff,\bff'\in \bfF$, $j,j' \in [n]$ and,
for all $\bfe\in\bfE$ and $\bff \in \bfF$,
\begin{equation}\label{def:relation}
  [x_\bfe,y_\bff] = \begin{cases} z_i & \textup{ if $\bff-\bfe$ is the $i$th
      standard basis vector of $\Z^n$,}\\ 0 & \textup{
      else.} \end{cases}
\end{equation}
Clearly $Z(L_{m,n}) = L_{m,n}':= [L_{m,n},L_{m,n}] = \la z_1,\dots,z_n
\ra$. In particular, $L_{m,n}$ is nilpotent of class~$2$. Setting
$$ e(m,n) = \# \bfE = \binom{n+m-2}{n-1}, \quad f(m,n) = \# \bfF = \binom{n+m-1}{n-1},$$
and further
 $$d(m,n) = f(m,n) + e(m,n), \quad h(m,n) = d(m,n) + n,$$ we find that
$$\rk_{\Z}(L_{m,n}/ L_{m,n}') = d(m,n),\quad \rk_{\Z}(L_{m,n}') =
\rk_{\Z}(Z(L_{m,n})) = n, \quad \rk_{\Z}(L_{m,n}) = h(m,n).$$
We recall further from \cite{SV1/15} the definition of the \emph{Igusa
  zeta function of degree $n$}
\begin{align}
  I_n(Y; \bfX) &=\sum_{I\subseteq \{1,\dots,n\}} \binom{n}{I}_{Y}
                 \prod_{i\in I}\frac{X_i}{1-X_i} %\nonumber\\
               =\frac{1}{1-X_n} \sum_{I\subseteq \{1,\dots,n-1\}} \binom{n}{I}_{Y}
                 \prod_{i\in I}\frac{X_i}{1-X_i}\nonumber\\
               &= \frac{\sum_{w\in S_n}
                 Y^{\ell(w)} \prod_{j\in \Des(w)}
                 X_j}{\prod_{i=1}^n(1-X_i)}\;\in\Q(Y,\bfX),\label{def:igusa}
\end{align}
where $\ell$ denotes the classical Coxeter length function on the
Coxeter group $S_n$ and the descent statistic $\Des$ is defined via
$\Des(w) = \{i\in \{1,\dots,n-1\} \mid w(i+1)<w(i) \}$. Moreover, for
$I=\{i_1,\dots,i_l\}\subseteq\{1,\dots,n\}_<$, the associated
\emph{Gaussian multinomial} is the polynomial
\begin{equation*}\label{def:gaussian.multi} \binom{n}{I}_Y = \binom{n}{i_{l}}_Y
  \binom{i_{l}}{i_{l-1}}_Y \cdots \binom{i_2}{i_1}_Y \in \Z [Y],
\end{equation*}
defined in terms of the \emph{Gaussian binomials} \begin{equation*}\label{def:gauss}
  \binom{a}{b}_Y = \frac{\prod_{i=a-b+1}^a (1-Y^i)}{\prod_{i=1}^b
    (1-Y^i)}\in \Z[Y]
\end{equation*}
for $a,b\in\N_0$ with $a\geq b$. It is well known that the ideal zeta
function $\zeta_{\lri^d}(s) = \zeta^{\nl}_{\lri^d}(s)$ of the abelian
$\lri$-Lie algebra $\lri^d$, enumerating all $\lri$-submodules of
finite index in $\lri^d$, is given by
\begin{equation}\label{equ:abel}
  \zeta_{\lri^d}(s) = \frac{1}{\prod_{i=0}^{d-1}(1-q^{i-s})} =
  I_d(q^{-1};(q^{(d-i)(i-s)})_{i=d-1}^0);
\end{equation}
cf., for instance, the first entry in the table on p.~218 of
\cite{GSS/88} (with $m=1$ and $r=0$). The main result of this paper is
the following.

\begin{theorem}\label{thm:main}
  For any compact discrete valuation ring $\lri$, with residue field
  cardinality $q$,
  \begin{equation}\label{equ:equ.main}
    \zeta^{\nl}_{L_{m,n}(\lri)}(s) = \zeta_{\lri^{d(m,n)}}(s) \cdot
    I_n\left(q^{-1};\left( q^{a^{\nl}_i(m,n) -s\, b^{\nl}_i(m,n)}
    \right)_{i=n-1}^0\right),\end{equation} where, for
  $i\in\{0,1,\dots,n-1\}$,
\begin{align}
  a^{\nl}_i(m,n) &= (n-i)(i + d(m,n)),\nonumber \\
  b^{\nl}_i(m,n) &= n - i + e(m,n) + \sum_{j=i+1}^n
                   e(m,j).\label{num.data}% could be one line
\end{align}
\end{theorem}

\begin{example}\label{exa:23}
  For $(m,n)=(2,3)$ we obtain
  $$\zeta^{\nl}_{L_{2,3}(\lri)}(s) = \frac{1+q^{9-7s} +q^{10-7s}+
    q^{18-10s} +q^{19-10s} +
    q^{28-17s}}{\left(\prod_{i=0}^8(1-q^{i-s})\right)(1-q^{27-12s})(1-q^{20-10s})(1-q^{11-7s})}.$$
  This formula exemplifies what seems to be a general phenomenon:
  cancellations as observed in \eqref{equ:abel} do not appear to
  affect the Igusa function of degree $n$ occurring in
  \eqref{equ:equ.main}.
\end{example}

\subsection{Mal'cev's correspondence}
Nilpotent Lie rings play an important role in the theory of finitely
generated, torsion-free nilpotent groups. Indeed, by the Mal'cev
correspondence, each such group $G$ has an associated nilpotent Lie
ring $\textup{L}_G$ such that, for almost all rational primes~$p$, the
ideal zeta function $\zeta^{\nl}_{\textup{L}_G(\Zp)}(s)$ is equal to
the \emph{local normal subgroup zeta function}
$\zeta^{\nl}_{G,p}(s) := \sum_{H \nl_p G} |G:H|^{-s}$, enumerating the
normal subgroups of $G$ of finite $p$-power index; cf.\
\cite[Theorem~4.1]{GSS/88}. In nilpotency class two, one may turn
instead to the Lie ring $L_G := Z(G) \oplus G/(Z(G))$. It is not hard
to show that the identity
$\zeta^{\nl}_{L_G(\Zp)}(s) = \zeta^{\nl}_{G,p}(s)$ holds for
\emph{all} primes~$p$. Every class-$2$-nilpotent Lie ring arises in
this way. Theorem~\ref{thm:main} thus yields, as a corollary, all
local normal subgroup zeta functions of the nilpotent groups
$\Delta_{m,n}$ associated to the Lie rings
$L_{m,n} = L_{\Delta_{m,n}}$, and thus the (global) \emph{normal zeta
  function}
$$\zeta^{\nl}_{\Delta_{m,n}}(s) := \sum_{H \nl \Delta_{m,n}}|
\Delta_{m,n}:H|^{-s} = \prod_{p \textup{
    prime}}\zeta^{\nl}_{\Delta_{m,n},p}(s) = \prod_{p \textup{
    prime}}\zeta^{\nl}_{L_{m,n}(\Zp)}(s) = \zeta^{\nl}_{L_{m,n}}(s).$$
It was in this context of subgroup growth of finitely generated
nilpotent groups that ideal zeta functions of Lie rings were first
studied systematically.

We are not aware of any subgroup-growth-theoretic interpretation of
the ideal zeta functions $\zeta^{\nl}_{L_G(\lri)}(s)$ if $\lri$ is not
of characteristic zero with prime residue field, i.e.\ of the form
$\lri=\Zp$.

\subsection{Previous work}
Theorem~\ref{thm:main} was previously known in a number of extremal
cases.

\begin{example}\label{exa:heisenberg}
  For $n=1$, the Lie rings $L_{m,1}$ are all isomorphic to the
  \emph{Heisenberg Lie ring}
$$\mfh = \la x,y,z \mid [x,y]=z, [x,z]=[y,z]=0\ra_{\Z}.$$
  Theorem~\ref{thm:main} thus confirms and extends the well-known,
  prototypical formula
  \begin{equation}\label{eq:heisenberg}
    \zeta^{\nl}_{\mfh(\lri)}(s) =
    \frac{1}{(1-q^{-s})(1-q^{1-s})(1-q^{2-3s})};
  \end{equation}
    cf.\ \cite[Proposition~8.1]{GSS/88}. For $n>1$, however, the Lie
    rings $L_{m,n}$ are pairwise non-isomorphic. For practical
    purposes, we may thus restrict attention to $n\geq 2$ as in
    \cite{BermanKlopschOnn/18}.
\end{example}

\begin{example}\label{exa:indec}
  For $n=2$, the Lie rings $L_{m,2}$ are the Lie rings associated, by
  the Mal'cev correspondence, to the indecomposable
  $\mathfrak{D}^*$-groups of odd Hirsch length $2m+3$ featuring in the
  classification up to commensurability of the finitely generated
  class-$2$-nilpotent groups with $2$-dimensional centre, developed in
  \cite{GSegal/84}. Their local normal zeta functions, and thus the
  ideal zeta functions of the Lie rings $L_{m,2}(\Zp)$, are
  essentially computed in \cite[Proposition~2]{Voll/04}, in agreement
  with the relevant special cases of Theorem~\ref{thm:main}.
\end{example}

\begin{example}\label{exa:grenham}
  For $m=1$, the Lie rings $L_{1,n}$ are the Lie rings associated to
  the so-called \emph{Grenham groups} $G_{n+1}$. These
  class-$2$-nilpotent groups have presentations
\begin{equation*}\label{equ:grenham.pres}
  G_{n+1} = \la x,y_1,\dots,y_n, z_1,\dots,z_n \mid \forall
  i\in\{1,\dots,n\}:\;[x,y_i]=z_i, \textup{ $z_i$ central}\ra.
\end{equation*}
The normal zeta functions of these groups, and thus the ideal zeta
functions of the Lie rings $L_{1,n}(\Zp)$, are computed in
\cite[Theorem~5]{Voll/05}.
\end{example}

\subsection{Related work}
The paper \cite{BermanKlopschOnn/18}, which introduced the groups
$\Delta_{m,n}$, computes their \emph{pro\-isomorphic zeta functions}
$\zeta^{\wedge}_{\Delta_{m,n}}(s)$, enumerating the subgroups of
finite index in $\Delta_{m,n}$ whose profinite completions are
isomorphic to that of~$\Delta_{m,n}$. By general principles, these
zeta functions also satisfy Euler product decompositions indexed by
the rational primes, whose factors are rational functions in
$p^{-s}$. In the notation of the current paper,
\cite[Theorem~1.4]{BermanKlopschOnn/18} establishes that the Euler
factor $\zeta^{\wedge}_{\Delta_{m,n},p}(s)$ of
$\zeta^{\wedge}_{\Delta_{m,n}}(s)$ at a rational prime~$p$,
enumerating the relevant subgroups of $\Delta_{m,n}$ of $p$-power
index, is of the form
$$\zeta^{\wedge}_{\Delta_{m,n},p}(s) = I_1(p^{-1}; z_{n+1}) \cdot
I_n(p^{-1}; (z_i)_{i=1}^n)$$ for explicitly given ``numerical data''
$z_i = p^{a^{\wedge}_i(m,n)-s\ b_i^{\wedge}(m,n)}$, for integers
$a^{\wedge}_i(m,n)$, $b_i^{\wedge}(m,n)$, comparable to (but different
from) those given in~\eqref{num.data}.

The \emph{subgroup zeta functions} $\zeta_{\Delta_{m,n}}(s) = \sum_{H
  \leq \Delta_{m,n}} | \Delta_{m,n}:H|^{-s}$, enumerating \emph{all}
finite index subgroups of the groups $\Delta_{m,n}$, are known
explicitly only for $m=1$, i.e.\ the Grenham groups from
Example~\ref{exa:grenham}; cf.\ \cite{VollBLMS/06}.

For a nonzero prime ideal $\mfp$ in a number ring $\Gri$, the
$\Gri_{\mfp}$-ideal zeta functions
$\zeta^{\nl}_{L_{m,n}(\Gri_{\mfp})}(s)$ should not be confused with
the $\Zp$-ideal zeta functions of $L_{m,n}(\Gri_{\mfp})$, considered
as $\Zp$-Lie algebras. For $n=1$, i.e.\ in the case of the Heisenberg
Lie ring $\mfh \cong L_{m,1}$ (see Example~\ref{exa:heisenberg}), the
latter have been computed for primes $p$ which are unramified in
$\Gri$ in \cite{SV1/15} and for primes $p$ which are non-split in
$\Gri$ in \cite{SV2/16}. In the paper \cite{CSV/19} we generalize
these computations to cover the ideal zeta functions of algebras
arising from a large class of Lie rings, including the Grenham Lie
rings $L_{1,n}$, via base extensions with various compact discrete
valuation rings. The paper gives a survey of applications of Igusa
functions in the area of zeta functions of groups and rings, and is
built on a generalization of Igusa functions.

\subsection{Organization and notation}

We prove Theorem~\ref{thm:main} in Section~\ref{sec:main.thm}, using
the general method introduced in \cite{Voll/05}. In
Section~\ref{sec:cor.por} we collect a number of corollaries and
porisms, notably pertaining to global analytic properties of ideal
zeta functions, functional equations satisfied by local ideal zeta
functions, their behaviour at zero, topological and reduced ideal zeta
functions, graded ideal zeta functions, and the representation zeta
functions associated to the groups $\Delta_{m,n}$.

We write $\N= \{1,2,\dots\}$ and $X_0= X \cup \{0\}$ for a subset
$X\subseteq \N$. Given $n\in\N_0$, we write $[n]=\{1,2,\dots,n\}$. The
notation $I = \{i_1,\dots,i_\ell\}_<\subseteq \N_0$ indicates that
$i_1 < i_2 < \dots < i_\ell$.

We write $\Mat_{a,b}(R)$ for the set of $a\times b$ matrices over a
ring~$R$. The ring's units are denoted by $R^*$. We write
$\Mat_{a}(R)$ instead of $\Mat_{a,a}(R)$.  Given matrices
$A_1,\dots,A_n$ with the same number of rows, we write $\left( A_1
\mid \dots \mid A_n\right)$ for their juxtaposition (or
concatenation). We write $\Id_{n}$ for the $n\times n$-identity matrix
and $\bfz_{a,b}$ for the zero matrix
$(0)_{ij}\in\Mat_{a,b}(R)$. Sometimes we write $\bfz$ for a zero
matrix whose dimensions are clear from the context.

We denote by $\lri$ a compact discrete valuation ring of arbitrary
characteristic, with maximal ideal $\mfp$, uniformizer $\pi \in
\mfp\setminus \mfp^2$, and residue field cardinality $q$. The
$\mfp$-adic valuation on $\lri$ will be denoted by $\val_{\mfp}$.

Given a property $P$, the ``Kronecker delta'' $\delta_P$ is equal to
$1$ if $P$ holds and equal to $0$ otherwise.

\section{Proof of Theorem~\ref{thm:main}}\label{sec:main.thm}

We maintain, to a large extent, the notation of \cite{Voll/05}.
Throughout, $m,n\in\N$ with $n\geq 2$ are arbitrary but fixed.

\subsection{Commutator matrix}
A key object in the computation of various zeta functions associated
to the Lie ring $L_{m,n}$ is its \emph{commutator matrix} with respect
to a $\Z$-basis. Consider the ordered (!) $\Z$-basis
$$\mcB_{m,n} = (x_{\bfe}, y_{\bff}, z_1,\dots,z_n)_{\bfe \in \bfE, \,\bff \in \bfF}$$
of $L_{m,n}$, where both the elements $x_\bfe$ and $y_{\bff}$ are
given, respectively, in \emph{reverse lexicographical ordering}, viz.\
the ordering obtained from the usual lexicographical orderings on
$\bfE$ resp.\ $\bfF$ but read backwards; cf.\
Example~\ref{exa:examples}~\eqref{exa23}.

Recall that the commutator matrix $M_{m,n}$ of $L_{m,n}$ with respect
to $\mcB_{m,n}$ is given as follows. For $i,j\in[d(m,n)]$ and
$w_i,w_j\in\mcB_{m,n}$, write
$[w_i,w_j] = \sum_{k=1}^n \lambda_{ij}^k z_k$, for structure
constants~$\lambda_{ij}^k\in\Z$. Then set $\bfY=(Y_1,\dots,Y_n)$ and
$$M_{m,n}(\bfY) = \left( \sum_{k=1}^n \lambda_{ij}^kY_k \right)_{ij} \in \Mat_{d(m,n)}(\Z[\bfY]).$$

To give a general, explicit description of $M_{m,n}$ we introduce the
following notation. For $e\in\N$ and a variable $Y$, write
$$\bfS_e(Y) = Y \cdot \Id_e \in\Mat_e(\Z[Y])$$
for the generic $(e\times e)$-scalar matrix. Set
$\bfY'=(Y_2,\dots,Y_n)$. We set $\bfB_{m,1}(Y) = Y\in\Mat_1(\Z[Y])$
and define recursively, for $n\geq 2$,

\begin{multline}\label{equ:coma}
  \bfB_{m,n}(\bfY) = \left( \begin{matrix} \bfS_{e(1,n-1)}(Y_1)&&&& \\
      \bfB_{1,n-1}(\bfY') & \bfS_{e(2,n-1)}(Y_1) & &&\\
      &\bfB_{2,n-1}(\bfY') &\ddots&&\\
      &&\ddots&\bfS_{e(m-1,n-1)}(Y_1)&\\
      &&&\bfB_{m-1,n-1}(\bfY')&\bfS_{e(m,n-1)}(Y_1)\\
      &&&&\bfB_{m,n-1}(\bfY')
\end{matrix} \right).
\end{multline}
Note that $\bfB_{1,n}(\bfY) =
(Y_1,\dots,Y_n)^{\trans}\in\Mat_{n,1}(\Z[\bfY])$.  Moreover, one
checks easily that
$$\bfB_{m,n}(\bfY) \in\Mat_{f(m,n),e(m,n)}(\Z[\bfY]),$$ say by using
parts~(1) and~(2) of the following lemma.

\begin{lemma}\label{lem:aux}\
\begin{enumerate}
 \item $\sum_{j=1}^m e(j,n-1) = e(m,n)$,
 \item $e(m,n) + f(m,n-1) = f(m,n)$,
 \item $\sum_{j=1}^ne(m,j) = f(m,n)$
\end{enumerate}
\end{lemma}
\begin{proof}
  Trivial. (One may want to use parallel summation for (1) and (3);
  cf.\ \cite[p.~174]{GKP/94}.)
\end{proof}

Various examples of $\bfB_{m,n}$ are given in
Example~\ref{exa:examples}. The following is evident.

\begin{lemma}\label{lem:coma}
If $\bfy = (y_1,\dots,y_n)\in\Fq^n \setminus \{ \mathbf{0} \}$,
then $\bfB_{m,n}(\bfy)$ has full rank~$e(m,n)$.
\end{lemma}

\begin{proposition}\label{pro:coma}
  The commutator matrix of $L_{m,n}$ with respect to the $\Z$-basis
  $\mcB_{m,n}$ is
$$M_{m,n}(\bfY) = \left( \begin{matrix}  & -\bfB_{m,n}(\bfY)^{\trans} \\ \bfB_{m,n}(\bfY) &  \end{matrix} \right)\in \Mat_{d(m,n)}(\Z[\bfY]).$$
\end{proposition}

\begin{proof}
  Given the defining relations \eqref{equ:def.rel} it is clear that
  the $(i,j)$-entry of $M_{m,n}(\bfY)$ vanishes if $i$ and $j$ are
  either both at most or both greater than~$e(m,n)$. The antisymmetry
  of $M_{m,n}(\bfY)$ is also evident, as $L_{m,n}$ is a Lie ring. To
  justify the specific shape of $\bfB_{m,n}(\bfY)$, recall that its
  columns are indexed by the generators $x_{\bfe}$, $\bfe \in \bfE$,
  whereas its rows are indexed by the generators $y_{\bff}$, $\bff \in
  \bfF$, of $L_{m,n}$. By definition of the commutator matrix
  $M_{m,n}(\bfY)$, the entry in position $(x_{\bfe},y_{\bff})$ is $Y_i
  \,\delta_{[x_\bfe,y_\bff]=z_i}$ for all $i\in[n]$;
  cf.\ \eqref{def:relation}.

  Crucially, both sets of generators are ordered
  reverse-lexicographically. Therefore, for $j=1,\dots,m$, the
  ``$j$-th column block'' $B_{m,n}^{(j)}(\bfY)$ of~$B_{m,n}(\bfY)$,
  comprising columns numbered
  $$\sum_{s< j}e(s,n-1)+1,\dots,\sum_{s\leq j}e(s,n-1),$$ echos the
  relations involving generators $x_\bfe$ indexed by elements
  $\bfe\in\N_0^n$ with first coordinate $m-j$, i.e.\ of the form
  $$\bfe = (m-j,\bfe') \textup{ for some } \bfe'\in \bfE(j,n-1).$$
  Likewise, for $i=1,\dots,m+1$, the ``$i$-th row block'' of
  $B_{m,n}(\bfY)$, comprising rows numbered
$$\sum_{r<i}e(r,n-1)+1,\dots,\sum_{r \leq i}e(r,n-1),$$ echos the
  relations involving generators $y_{\bff}$ indexed by elements
  $\bff\in\N_0^n$ with first coordinate $m-i+1$, i.e.\ of the form
  $$\bff = (m-i+1,\bff') \textup{ for some } \bff'\in \bfF(i-1,n-1).$$

  We describe in detail the submatrices of the column block
  $B_{m,n}^{(j)}(\bfY)$ defined by its intersection with the $i$-th
  row blocks of $B_{m,n}(\bfY)$, thereby justifying the claim that
\begin{equation*}\label{equ:col.block}
  B_{m,n}^{(j)}(\bfY) = \left( \begin{matrix} \bfz_{e(j-1,n),e(j,n-1)}
    \\ \bfS_{e(j,n-1)}(Y_1)\\ \bfB_{j,n-1}(\bfY')
    \\ \bfz_{(f(m,n)-f(j,n)),e(j,n-1)} \end{matrix}\right).
\end{equation*}
\begin{enumerate}
 \item If $i < j$, then the relevant rows of $B^{(j)}_{m,n}(\bfY)$
   comprise the relations between generators indexed by elements of
   the form
    \begin{alignat*}{1}
      \bfe &= (m-j, \bfe'),\\ %& \quad e_2+\dots + e_n &= j-1\\
      \bff &= (m-i+1, \bff').%,& f_2+\dots + f_n &= i-1
    \end{alignat*}
  However, $(m-i+1) - (m-j) = j-i+1 \geq 2$, so
  $[x_{\bfe},y_{\bff}]=0$ for the relevant elements, and hence the
  relevant submatrix of $B^{(j)}_{m,n}(\bfY)$ is
  $\bfz_{e(j-1,n),e(j,n-1)}$.
 \item If $i=j$, then the relevant rows of $B^{(j)}_{m,n}(\bfY)$
   comprise the relations between generators indexed by elements of
   the form
    \begin{alignat*}{1}
      \bfe &= (m-j, \bfe'),\\%& \quad e_2+\dots + e_n &= j-1\\
      \bff &= (m-j+1, \bff').%,& f_2+\dots + f_n &= j-1
    \end{alignat*}
   As $[x_{\bfe},y_{\bff}]=z_1\delta_{\bfe'=\bff'}$, the relevant
   submatrix of $B^{(j)}_{m,n}(\bfY)$ is $\bfS_{e(j,n-1)}(Y_1)$.
 \item If $i = j+1$, then the relevant rows of $B^{(j)}_{m,n}(\bfY)$
   comprise the relations between generators indexed by elements of
   the form
    \begin{alignat*}{2}
      \bfe &= (m-j, \bfe')& \textup{ for some }\bfe'\in \bfE(j,n-1),\\ \bff &=
      (m-j, \bff'),&  \textup{ for some }\bff'\in \bfF(j,n-1).
    \end{alignat*}
    This justifies the claim that the relevant submatrix of
    $B^{(j)}_{m,n}(\bfY)$ is $B_{j,n-1}(\bfY')$.
\item If $i>j+1$, then  the relevant rows of $B^{(j)}_{m,n}(\bfY)$
   comprise the relations between generators indexed by elements of
   the form
    \begin{alignat*}{1}
      \bfe &= (m-j, \bfe'),\\
      \bff &= (m-i+1,\bff').
    \end{alignat*}
However, $(m - i + 1) - (m-j) = j-i+1 < 0$, so $[x_{\bfe},y_{\bff}]=0$
for the relevant elements, and hence the relevant submatrix of
$B^{(j)}_{m,n}(\bfY)$ is $\bfz_{(f(m,n)-f(j,n)),e(j,n-1)}$.\qedhere
%
%    But $m-j > m-i+1$, justifying our claim that the relevant submatrix of
%$B^{(j)}_{m,n}(\bfY)$ is $\bfz_{(f(m,n)-f(j,n)),e(j,n-1)}$.\qedhere
\end{enumerate}
\end{proof}

\begin{example}\label{exa:examples}\
  \be
\item For $n=2$, we obtain
$$\bfB_{m,2}(Y_1,Y_2) = \left( \begin{matrix} Y_1 &&& \\Y_2 & Y_1 &&
    \\ & Y_2 & \ddots && \\ && \ddots& Y_1 \\ &&&Y_2 \end{matrix}
\right) \in\Mat_{m+1,m}(\Z[Y_1,Y_2]).$$ Up to a simple reordering of
rows and columns and swapping the variables $Y_1$ and~$Y_2$, the
matrix $M_{m,2}(Y_1,Y_2)$ is the commutator matrix described in
\cite[Theorem~4]{Voll/04} (essentially \cite[Theorem~6.3]{GSegal/84})
associated to the {indecomposable $\mathfrak{D}^*$-group}
$\Delta_{m,2}$; cf.\ Example~\ref{exa:indec}.
\item For $m=1$, we obtain 
$$M_{1,n}(Y_1,\dots,Y_n) = \left( \begin{matrix} &-Y_1&-Y_2&
    \dots& -Y_n \\Y_1&&&&\\Y_2 &&&& \\ \vdots &&& \\
    Y_n&&& \end{matrix} \right) \in\Mat_{n+1}(\Z[Y_1,\dots,Y_n]),$$
the commutator matrix of the {Grenham Lie ring} $L_{1,n}$ with respect
to the $\Z$-basis $(x,y_1,\dots,y_n, z_1,\dots,z_n)$; cf.\
Example~\ref{exa:grenham}.
\item \label{exa23}  For $m=2$, $n=3$,
  $$\mcB_{2,3} = \left\{ e_{(1,0,0)}, e_{(0,1,0)},e_{(0,0,1)},
    \,
    f_{(2,0,0)},f_{(1,1,0)},f_{(1,0,1)},f_{(0,2,0)},f_{(0,1,1)},f_{(0,0,2)},\,
    z_1,z_2,z_3 \right\},$$ yielding
  $$\bfB_{2,3}(Y_1,Y_2,Y_3) = \left( \begin{array}{c|cc} Y_1 && \\ \hline Y_2 & Y_1 & \\ Y_3& &Y_1  \\ \hline  &Y_2& \\ &Y_3&Y_2\\ &&Y_3 \end{array} \right).$$ 
\ee
\end{example}

\subsection{Informal overview of the proof}
We use the general method introduced in \cite{Voll/05}. The fact that
there only the case $\lri=\Zp$ is treated explicitly is
inconsequential: all that is needed is the fact that $\Zp$ is a
compact discrete valuation ring.

According to \cite[Lemma~1]{Voll/05} (essentially
\cite[Lemma~6.1]{GSS/88}) there exists a rational function
$A^{\nl}_{m,n}$ in $q$ and $q^{-s}$ such that
 \begin{equation*}\label{equ:zeta.AJM}
   \zeta^{\nl}_{L_{m,n}(\lri)}(s) = \zeta_{\lri^{d(m,n)}}(s)
   \frac{1}{1-q^{d(m,n) n-s\, h(m,n)}}A^{\nl}_{m,n}(q,q^{-s});
\end{equation*}
cf.\ Remark~\ref{rem:gafa} below. The function $A^{\nl}_{m,n}$ may be
viewed as a generating function enumerating the values of two
integer-valued functions $w$ and $w'$ on the set $\mcV_n$ of homothety
classes of lattices in $Z(L_{m,n}(\lri)) \cong \lri^n$, viz.\ vertices
in the affine Bruhat-Tits building of type $\widetilde{A_{n-1}}$
associated to the group $\GL_n(k)$, where $k = \textup{Frac}(\lri)$ is
the field of fractions of the local ring~$\lri$:
\begin{equation}\label{equ:Amn}
 A^{\nl}_{m,n}(q,q^{-s}) = \sum_{[\Lambda'] \in \mcV_n} q^{d(m,n)
   w([\Lambda'])-s \,w'([\Lambda'])}.
\end{equation}
The function $w$ captures the ($\log_q$ of the) {index} of the maximal
integral element $\Lambda'_{\max}$ of $[\Lambda']$ in
$Z(L_{m,n}(\lri))$. It is a simple function of the {elementary
  divisors} of $\Lambda'_{\max}$ relative to $\lri^n$; cf.\
\eqref{equ:w}. The function $w'$ records, in addition, the index in
$L_{m,n}(\lri)$ of the lattice $X([\Lambda'])$ in $L_{m,n}(\lri)$,
defined by the condition
\begin{equation}\label{def:X}
  X([\Lambda'])/\Lambda'_{\max} = Z(L_{m,n}(\lri)/\Lambda'_{\max});
\end{equation}
  see~\eqref{def:w'}.  We will use the interpretation of this index as
  the index of the kernel of a system of linear congruences on
  $\lri^{d(m,n)}$ provided by~\cite[Theorem~6]{Voll/05}; see
  Proposition~\ref{lem:w'}.

The discussion so far applies, \emph{mutatis mutandis}, to (the
$\lri$-points of) any class-$2$-nilpotent Lie ring. In general, the
index of $X([\Lambda'])$ in the Lie ring's abelianization will depend
in an arithmetically subtle way on $[\Lambda']$. The key to the proof
of Theorem~\ref{thm:main} is the realization that, for the Lie rings
$L_{m,n}$, the index of $X([\Lambda'])$ depends solely, and in a
($\log_q$-)linear fashion, on the elementary divisors of
$\Lambda'_{\max}$. Consequently, the rational function $A^{\nl}_{m,n}$
may be expressed in terms of an Igusa function of degree $n$. In the
course of the proof of these facts we will compute the relevant
($\log_q$-)linear functionals explicitly, making heavy use of the
combinatorial description~\eqref{equ:coma} of the commutator
matrix~$M_{m,n}(\bfY)$.

\begin{remark}\label{rem:gafa}
  \cite[Lemma~1]{Voll/05} does not hold in the generality proclaimed
  in \cite{Voll/05}. To apply for a prime $p$, the centre of the
  reduction of $G_p$ modulo $p$ needs to coincide with the reduction
  modulo $p$ of the centre of~$G_p$. This condition is satisfied
  generically, but may fail for finitely many primes. (The main
  results of \cite{Voll/05} are unaffected by this, as they are only
  stated (and known to hold only) for almost all primes $p$.)  In the
  case of the groups $\Delta_{m,n}$, this set of exceptional primes is
  indeed empty, as one checks without difficulty, so the conclusion of
  \cite[Lemma~1]{Voll/05} applies for all primes~$p$. I am grateful to
  the referee for pointing out these facts.
\end{remark}

\subsection{Parametrizing lattices}\label{subsec:lattices}
We recall, e.g.\ from \cite{Voll/05}, a parametrization of maximal
integral lattices inside $\lri^n$. Let $\Lambda'\leq \lri^n$ be a
maximal $\lri$-sublattice, i.e.\ $\pi^{-1}\Lambda'\not\leq \lri^n$,
where $\pi\in \mfp \setminus \mfp^2$ is a uniformizer. The lattice
$\Lambda'$ is said to be of \emph{type}
$\nu(\Lambda') = (I, \bfr_{I})$ if
$$I = \{ i_1,\dots,i_\ell\}_< \subseteq [n-1], \quad \bfr_{I} = (r_{\iota})_{\iota\in I}\in \N^I,$$
and $\Lambda'$ has elementary divisors
\begin{equation}\label{eq:el.div}
\left( (1)^{(i_1)}, (\pi^{r_{i_1}})^{(i_2-i_1)}, \dots, (\pi^{\sum_{\iota\in I}r_\iota})^{(n-i_\ell)}\right) =: (\pi^\nu)\in\N^n.
\end{equation}
  Clearly $|\lri^d:\Lambda'| = q^{\sum_{\iota\in I} r_\iota(n-\iota)}$,
  whence, in the notation of \cite[Definition~2]{Voll/04},
\begin{equation}\label{equ:w}
w([\Lambda']) = \sum_{\iota\in I} r_\iota(n-\iota).
\end{equation}
It is well known and not hard to show (cf., for instance,
\cite[Lemma~2]{Voll/05}) that
\begin{equation}\label{equ:f}
  f_{I,\bfr_I}(q) := \#\{ \Lambda'\leq \lri^d \mid \Lambda' \textup{ maximal and of type $(I, \bfr_{I})$}\} = \binom{n}{I}_{q^{-1}} q^{\sum_{\iota\in I}r_\iota\iota(n-\iota)}.
\end{equation}
Here $\binom{n}{I}_{q^{-1}}$ is the value of the Gaussian multinomial
$\binom{n}{I}_Y\in\Z[Y]$ at $Y=q^{-1}$. Of central importance in the
following is the elementary fact that the group
$\Gamma_n := \GL_n(\lri)$ acts transitively on the set of maximal
lattices of given type $\nu$. Denoting, for $i\in[n]$, by
$\varepsilon_i$ the $i$-th standard basis vector of $\lri^n$, the
lattice
$$\Lambda'= \bigoplus_{i=1}^n(\pi^\nu)_i \lri \varepsilon_i$$ is
evidently of type $\nu$. The stabilizer subgroup $\Gamma_{\nu}$ of
$\Lambda'$ in $\Gamma_n$ is easily described explicitly, but we will
not need such a description. What we will need are two facts. First,
$\Gamma_\nu$ contains the (Borel) subgroup $B_n$ of lower-triangular
matrices in $\Gamma_n$. (Note that the matrix description of
$\Gamma_{\nu}$ on p.~1203 of \cite{Voll/10} in terms of block matrices
which are block upper-triangular modulo $p$ is given with respect to
the reverse ordering of the elementary divisors~\eqref{eq:el.div}.)
Second, the orbit-stabilizer-theorem gives us a bijection between
maximal lattices of type $\nu$ and cosets
in~$\Gamma_n/\Gamma_\nu$. Fix a coset $\alpha\Gamma_\nu$. We claim
that, after a permutation of the rows if necessary (corresponding to a
monomial change of $\lri$-basis for $\lri^n$), it contains a
representative of the form
$$\alpha_0 = \left( \begin{matrix} &&&\alpha_{1n}\\&& \alpha_{2\,n-1}
    & \alpha_{2\,n}\\ & \adots&&\vdots\\\alpha_{n1}& \dots &
    \alpha_{n\, n-1} & \alpha_{n\,n}\end{matrix}\right)\in\Gamma_n.$$
Indeed, if the $(1,n)$-entry of $\alpha_0$ is a unit (as we may assume
without loss of generality), we may use it to ``clear'' the remaining
entries in the first row of $\alpha_0$ by right-multiplication by a
suitable element of $B_n\leq \Gamma_{\nu}$. The claim follows
inductively.

We write $\alpha_0 = (\alpha^{(1)} \mid \dots \mid \alpha^{(d)})$,
i.e.\ $\alpha^{(j)}$ denotes the $j$-th column of $\alpha_0$. Note
that the antidiagonal entries of $\alpha_{0}$ are all units:
$\alpha_{j,n+1-j}\in\lri^*$ for all~$j\in[n]$.

\subsection{Solving linear congruences}
Let $\Lambda' \leq \lri^n$ be a maximal lattice of type $(I,\bfr_I)$,
corresponding to a coset $\alpha\Gamma_\nu$ as described in
Section~\ref{subsec:lattices}.  By \cite[Theorem~6]{Voll/05} the index
$|L_{m,n}(\lri): X([\Lambda'])|$ equals the index in $\lri^{d(m,n)}$
of the following system of linear congruences, where we write
$\bfg = (\bfg^{(1)},\bfg^{(2)})\in \lri^{e(m,n)}\times \lri^{f(m,n)}
\cong \lri^{d(m,n)}$:
\begin{equation}\label{equ:lincon}
 \forall j\in [n]:\; \bfg M_{m,n}(\alpha^{(j)}) \equiv 0 \bmod
 (\pi^{\nu})_j.
\end{equation}
Set $r := \sum_{\iota \in I} r_\iota$. 

\begin{lemma}\label{lem:bigmatrix}
  \eqref{equ:lincon} holds if and only if $\bfg^{(1)} \equiv 0 \bmod (\pi^r)$ and
 \begin{equation}\label{equ:bigmatrix}
\bfg^{(2)} \left( \pi^r \bfB_{m,n}(\alpha^{(1)}) \mid \dots \mid
\pi^{\sum_{\iota \geq j} r_\iota} \bfB_{m,n}(\alpha^{(j)}) \mid \dots
\mid \bfB_{m,n}(\alpha^{(n)})\right) \equiv 0 \bmod (\pi^r).
\end{equation}
\end{lemma}

\begin{proof}
  By multiplying each of the congruences in \eqref{equ:lincon} by the
  appropriate power of the uniformizer $\pi$ we may consider them all
  as congruences modulo $(\pi^r)$. By concatenating the relevant
  matrices, we obtain that \eqref{equ:lincon} is equivalent to the
  single congruence
\begin{equation}\label{equ:biggermatrix}
  \bfg \left( \pi^r M_{m,n}(\alpha^{(1)}) \mid \dots \mid \pi^{\sum_{\iota \geq j} r_\iota} M(\alpha^{(j)}) \mid \dots \mid M_{m,n}(\alpha^{(n)})\right) \equiv 0 \bmod (\pi^r).
\end{equation}
  Note that the vector $\alpha^{(n)}\in\lri^n$ is nonzero modulo
  $\mfp$. By Proposition~\ref{pro:coma} and Lemma~\ref{lem:coma},
$$\bfg M_{m,n}(\alpha^{(n)}) = \bfg \left( \begin{matrix} &
    -\bfB_{m,n}(\alpha^{(n)})^{\trans} \\ \bfB_{m,n}(\alpha^{(n)})
    & \end{matrix} \right) \equiv 0 \bmod (\pi^r)$$ thus only holds if
  $\bfg^{(1)}\equiv 0 \bmod (\pi^r)$. Deleting the first $e(m,n)$ rows
  from the matrix in \eqref{equ:biggermatrix} one sees that, in this
  case, \eqref{equ:lincon} is equivalent to \eqref{equ:bigmatrix}.
\end{proof}

Note that, in the $f(m,n)\times (e(m,n)\cdot n)$ matrix in
\eqref{equ:bigmatrix}, the first $i_1$ blocks
$\pi^r \bfB_{m,n}(\alpha^{(j)})$, $j\in[i_1]$, i.e.\ the first
$e(m,n)\cdot i_1$ columns, are of course redundant. Recall from
\cite[Def.~2]{Voll/04} that, with $X([\Lambda'])$ defined as
in~\eqref{def:X},
\begin{equation}\label{def:w'}
  w'([\Lambda']) = \log_q(|\lri^n:\Lambda'|) + \log_q(|L_{m,n}(\lri):X([\Lambda'])|).
\end{equation}

\begin{proposition}\label{lem:w'}
  The index of the lattice of elements $\bfg\in\lri^{d(m,n)}$
  satisfying the simultaneous congruences \eqref{equ:lincon} equals
$$q^{\sum_{\iota \in I}r_\iota\left(e(m,n) + \sum_{j=\iota+1}^n e(m,j)\right)}.$$
In other words,
$$w'([\Lambda']) = \sum_{\iota \in I}r_\iota\left( n - \iota + e(m,n) + \sum_{j=\iota+1}^n e(m,j)\right).$$
\end{proposition}

\begin{proof}
  For $j=1,\dots,n$, write
  $$B_j := \pi^{\sum_{\iota \geq j}r_\iota}
  \bfB_{m,n}(\alpha^{(j)})\in\Mat_{f(m,n), e(m,n)}(\lri)$$
  for the $j$-th column block of the matrix in~\eqref{equ:bigmatrix}.
  Note that $\val_\mfp(B_j) = \sum_{\iota\geq j}r_\iota \in\N_0$. Set
$$B := (B_1 \mid \dots \mid B_n) \in \Mat_{f(m,n), n e(m,n)}(\lri).$$
In the light of Lemma~\ref{lem:bigmatrix} we need to prove that the
index in $\lri^{f(m,n)}$ of the solutions of the congruence
$$\bfg^{(2)} B \equiv 0 \bmod (\pi^r)$$ equals $q^{\sum_{\iota\in I}
  r_\iota\left(\sum_{j=\iota+1}^n e(m,j)\right)}$. For this it suffices to show
that
\begin{equation}\label{equ:equiv}
  B \textup{ is equivalent to }(\bfz_{f(m,n),ne(m,n)-f(m,n)}
  \mid D_{f(m,n)})\in\Mat_{f(m,n),ne(m,n)}(\lri),
  \end{equation}where
$$D_{f(m,n)} := \left(\begin{matrix}&&&&\Id_{e(m,n)} %\\&&& \adots &
  \\ && \pi^{\sum_{\iota\geq n-1}r_\iota} \Id_{e(m,n-1)} && \\ & \adots&
  &&\\ \pi^r \Id_{e(m,1)}&&&&\end{matrix} \right)\in
\Mat_{f(m,n)}(\lri).$$
We proceed inductively, replacing $B$
successively by equivalent matrices.

We first note that the top $e(m,n)$ rows of $B_n$ form a matrix
$\wt{B} \in \GL_{e(m,n)}(\lri)$, as $\alpha_{1n}\in \lri^*$. We use
$\wtB$ to clear---by suitable column operations---all other entries in
the top $e(m,n)$ rows of $B$. Note that this does not affect the last
$e(m,j)$ columns in either of the matrices $B_j$, $j=1,\dots,n-1$, nor
the valuations of these matrices. We now use $\wtB$ to clear---by
suitable row operations---all entries of $B$ below $\wtB$, leaving the
other columns unaffected. We may then also assume that $\wtB =
\Id_{e(m,n)}$.

We have thus replaced $B$ by an equivalent matrix of the form
$$\left(\begin{matrix} & & & \Id_{e(m,n)}\\B_1'& \cdots &
  B_{n-1}'& \end{matrix}\right),$$ where, for each $j=1,\dots,n-1$,
the matrix $B_j'$ has valuation $\sum_{\iota\geq j}r_\iota$ and the
matrices $\left(\begin{matrix} \bfz \\ B_j'\end{matrix} \right)$ and
  $B_j\in\Mat_{f(m,n),e(m,n)}(\lri)$ coincide in their last $e(m,j)$
  columns. Set
  $$B'= (B_1'\mid \dots \mid B_{n-1}') \in
  \Mat_{f(m,n-1),(n-1)e(m,n)}(\lri).$$ The top $e(m,n-1)$ rows and
  last $e(m,n-1)$ columns of $B_{n-1}'$ form a matrix
  $\pi^{\sum_{\iota\geq n-1}r_\iota} \wt{\wtB}$ for
  $\wt{\wtB}\in\GL_{e(m,n-1)}(\lri)$. We may use it to clear all other
  entries in the top $e(m,n-1)$ rows of $B'$. Note that this does not
  affect the last $e(m,j)$ columns in either of the matrices $B_j'$,
  $j=1,\dots,n-2$, nor the valuations of these matrices. We now use
  $\pi^{\sum_{\iota\geq n-1}r_\iota} \wt{\wtB}$ to clear all entries
  of $B'$ below~$\wt{\wtB}$, leaving the other columns unaffected. We
  may then also assume that $\wt{\wtB}= \Id_{e(m,n-1)}$.

We have thus replaced $B'$ by an equivalent matrix of the form
  $$\left(\begin{matrix} & & & \pi^{\sum_{\iota\geq
      n-1}r_\iota}\Id_{e(m,n-1)}\\B_1''& \cdots &
  B_{n-2}''& \end{matrix}\right),$$ where, for $j=1,\dots,n-2$, the
matrices $B_j''\in\Mat_{f(m,n-2),e(m,n)}(\lri)$ each have valuation
$\sum_{\iota \geq j} r_\iota$ and the matrices $\left(\begin{matrix}
  \bfz \\ B_j''\end{matrix} \right)$, $\left(\begin{matrix} \bfz
    \\ B_j'\end{matrix} \right)$, and
    $B_j\in\Mat_{f(m,n),e(m,n)}(\lri)$ coincide in their last $e(m,j)$
    columns.

      The claim \eqref{equ:equiv} follows by continuing inductively in
      this manner.
\end{proof}

\subsection{Completion of the proof of Theorem~\ref{thm:main}}\label{subsec:completion}

We are now ready to complete the computation of the rational function
$A^{\nl}_{m,n}(q,q^{-s})$ featuring in~\eqref{equ:zeta.AJM}. Indeed,
using \eqref{equ:Amn}, \eqref{equ:w}, \eqref{equ:f}, and
Proposition~\ref{lem:w'}, we obtain
\begin{align*}
  {A^{\nl}_{m,n}(q,q^{-s})} &= \sum_{[\Lambda'] \in \mcV_n}
  q^{d(m,n) w([\Lambda'])-s\,w'([\Lambda'])} \\ &= \sum_{I \subseteq
    [n-1]} \sum_{\bfr_I \in \N^I} f_{I,\bfr_I}(q)
  q^{\sum_{\iota\in I} r_\iota\left((n-\iota)d(m,n)-s \left(n-\iota +
    e(m,n) + \sum_{j=\iota+1}^n e(m,j)\right)\right)}\\ &= \sum_{I
    \subseteq [n-1]} \binom{n}{I}_{q^{-1}} \sum_{\bfr_I\in \N^I}
  q^{\sum_{\iota\in I}r_\iota \left((n-\iota) \left( \iota +
    d(m,n)\right)-s\left(n - \iota + e(m,n) + \sum_{j=\iota+1}^n
    e(m,j) \right)\right)} \\ &= \sum_{I \subseteq [n-1]}
  \binom{n}{I}_{q^{-1}} \prod_{i\in I} \frac{q^{a^{\nl}_i(m,n)-s\,b^{\nl}_i(m,n)}}{1- q^{a^{\nl}_i(m,n)-s\, b^{\nl}_i(m,n)}},
\end{align*}
with $a^{\nl}_i(m,n)$ and $b^{\nl}_i(m,n)$ defined as in
\eqref{num.data}.  Using Lemma~\ref{lem:aux} (3) one easily computes
$(a_0^{\nl}(m,n),b_0^{\nl}(m,n)) = (d(m,n)n,h(m,n))$, whence, using
\eqref{def:igusa}, we obtain that
$$\frac{1}{1-q^{d(m,n)n -s\, h(m,n)}} A^{\nl}_{m,n}(q,q^{-s}) = I_n\left(q^{-1};\left( q^{a^{\nl}_i(m,n) -s\, b^{\nl}_i(m,n)} \right)_{i=n-1}^0\right).$$
Theorem~\ref{thm:main} follows now from \eqref{equ:zeta.AJM}.

\section{Corollaries and porisms}\label{sec:cor.por}
We record a few consequences of Theorem \ref{thm:main} and its
proof. Throughout, $\lri$ denotes, as before, a compact discrete
valuation ring. Let $\Gri$ be the ring of integers of a number field
$K$, with Dedekind zeta function $\zeta_K(s)$. We set
$L_{m,n}(\Gri) := L_{m,n} \otimes_\Z\Gri$.

\subsection{Global analytic properties}\label{subsec:global}

\begin{corollary} The ideal zeta function
  $\zeta^{\nl}_{L_{m,n}(\Gri)}(s)$ has abscissa of convergence
  $\alpha^{\nl}(m,n) = d(m,n)$ and allows for meromorphic continuation
  to (at least) the complex half-plane
  $$\left\{ s\in\C \mid \Re(s) > \beta^{\nl}(m,n) \right\},$$ where
  $$\beta^{\nl}(m,n) := \max \left\{
    \frac{a_i^{\nl}(m,n)-1}{b_i^{\nl}(m,n)}\mid
    i=0,\dots,n-1\right\},$$ and even the whole complex plane if
  $n\leq 2$. In any case, the continued function has a simple pole at
  $s=\alpha^{\nl}(m,n)$.
\end{corollary}

\begin{proof}
  It is well known (see \eqref{equ:abel}) that
  $\zeta_{\Gri^d}(s) = \prod_{i=0}^{d-1} \zeta_K(s-i)$, has abscissa
  of convergence $s=d$, and admits meromorphic continuation to the
  whole complex plane to a function that has a simple pole at
  $s=d$. It thus suffices to note that the Euler product
  \begin{equation}\label{equ:euler}
    \prod_{\mfp \in \Spec(\Gri)\setminus \{ (0) \}} I_n\left(q_{\mfp}^{-1};
      \left(q_{\mfp}^{a^{\nl}_i(m,n)-s\,b_i^{\nl}(m,n)}\right)_{i=n-1}^0\right)
    \end{equation}
    has 
\begin{itemize}
\item[(A)] abscissa of convergence
  $\max \left\{ \frac{a_i^{\nl}(m,n)+1}{b_i^{\nl}(m,n)}\mid
    i=0,\dots,n-1\right\} < d(m,n)$
  and 
\item[(B)] meromorphic continuation to
  $\left\{ s\in\C \mid \Re(s) > \beta^{\nl}(m,n) \right\}$.
\end{itemize}

To verify (A) we observe that the Euler factors of \eqref{equ:euler}
may be written in the form
  \begin{equation}\label{eq:frac}\frac{\sum_{w\in S_n}
      q_{\mfp}^{-\ell(w)}\prod_{j\in\Des(w)}q_{\mfp}^{a^{\nl}_j(m,n)-s\,b_j^{\nl}(m,n)}}{\prod_{i=0}^{n-1}\left(1-q_{\mfp}^{a^{\nl}_i(m,n)-s\,b_i^{\nl}(m,n)}\right)}.\end{equation}
  Both numerator and denominator of this expression are given by
  bivariate polynomial expressions in $q_{\mfp}$ and $q_{\mfp}^{-s}$
  with integer coefficients. The abscissa of convergence of
  the Euler product
$$\prod_{\mfp \in \Spec(\Gri)\setminus \{ (0) \}}\frac{1}{\prod_{i=0}^{n-1}\left(1-q_{\mfp}^{a^{\nl}_i(m,n)-s\,b_i^{\nl}(m,n)}\right)}$$ arising from the denominators of~\eqref{eq:frac} is $\max \left\{ \frac{a_i^{\nl}(m,n)+1}{b_i^{\nl}(m,n)}\mid
  i=0,\dots,n-1\right\}$.
We omit the elementary proof of the fact that this quantity is
dominated by $d(m,n)$. It is a simple exercise to check that it
dominates the abscissa of convergence of the Euler product
$$\prod_{\mfp \in \Spec(\Gri)\setminus \{ (0) \}}\sum_{w\in S_n} q_{\mfp}^{-\ell(w)}\prod_{j\in\Des(w)}q_{\mfp}^{a^{\nl}_j(m,n)-s\,b_j^{\nl}(m,n)}$$
over the numerators of~\eqref{eq:frac}. The latter is given, for
instance, by the formula in \cite[Lemma~5.4]{duSWoodward/08}.

To verify claim (B), we employ \cite[Lemma~5.5]{duSWoodward/08} and
note that
  $$\max \left\{ \frac{-\ell(w) + \sum_{j\in \Des(w)}
      a^{\nl}_j(m,n)}{\sum_{j\in \Des(w)} b^{\nl}_j(m,n)} \mid w\in
    S_n\setminus\{e\}\right\}$$ is attained at one of the elements
    $w\in S_n$ with $\# \Des(w)=1$. 

    The stronger claim for $n=2$ follows from the observation that the
    Euler product \eqref{equ:euler} is
  \begin{multline*}
    \prod_{\mfp\in\Spec(\Gri)\setminus\{(0)\}} \frac{1 +
      q_{\mfp}^{(1-s)(2m+1)}}{(1-q_{\mfp}^{2m+2-s(2m+1)})(1-q_\mfp^{2(2m+1)-s(2m+3)})}
    =
    \\\frac{\zeta_K((2m+1)s-2m-2)\zeta_K((2m+3)s-2(2m+1))\zeta_K((s-1)(2m+1))}{\zeta_K((s-1)(4m+2))}.
    \end{multline*}
For $n=1$ it follows from the fact that
$\zeta^{\nl}_{L_{m,1}(\Gri)}(s) =
\zeta_K(s)\zeta_K(s-1)\zeta_K(3s-2)$; see \eqref{eq:heisenberg}.
\end{proof}

\begin{remark}
  It remains an interesting challenge to determine the maximal domain
  of meromorphicity of the global ideal zeta functions
  $\zeta^{\nl}_{L_{m,n}(\Gri)}(s)$ for general $m$ and $n$. The good
  analytic properties for $n\leq 2$ are, in any case, exceptional: for
  $n>2$, the numerator of an Igusa function of degree $n$ will not, in
  general, factor nicely; see, for instance, Example~\ref{exa:23}
  (where we obtain
  $\beta^{\nl}(2,3) = \max\{\frac{11-1}{7},\frac{20-1}{10}\} =
  \frac{19}{10} < 9 = \alpha^{\nl}(2,3)$).
\end{remark}

\subsection{Local functional equations}

\begin{corollary}
$$\left. \zeta^{\nl}_{L_{m,n}(\lri)}(s)\right|_{q \rightarrow q^{-1}}
  = (-1)^{h(m,n)}
  q^{\binom{h(m,n)}{2}-s\,(d(m,n)+h(m,n))}\zeta^{\nl}_{L_{m,n}(\lri)}(s).$$
\end{corollary}
\begin{proof}
  Cf.\ \cite[Theorem~4]{Voll/05}.
\end{proof}
For \emph{almost all} residue field characteristics, these functional
equations had been established, in greater generality,
in~\cite[Theorem~C]{Voll/10}; see also \cite[Theorem~1.2 and
Corollary~1.3]{Voll/17}.

\subsection{$\mfp$-Adic behaviour at zero}
Rossmann has put forward the remarkable expectation that quite general
local zeta functions associated with nilpotent algebras of
endomorphisms should have predictable behaviour at~$s=0$. The
following consequence of Theorem~\ref{thm:main} establishes
\cite[Conjecture~IV ($\mathfrak{P}$-adic form)]{Rossmann/15} in the
relevant special cases.

\begin{corollary}\label{cor:pad}
$$\left.\frac{\zeta^{\nl}_{L_{m,n}(\lri)}(s)}{\zeta_{\lri^{h(m,n)}}(s)}\right|_{s=0}
  = 1$$
\end{corollary}

\begin{proof} Note that both $\zeta^{\nl}_{L_{m,n}(\lri)}(s)$ and
  $\zeta_{\lri^{h(m,n)}}(s)$ have a simple pole at $s=0$.  By
  \eqref{equ:abel} it suffices to observe that
$$I_n\left(q^{-1};\left( q^{(n-i)(i + d(m,n))}\right)_{i=n-1}^0 \right) =\zeta_{\lri^n}(-d(m,n)) = \frac{1}{\prod_{i=0}^{n-1}(1-q^{d(m,n)+i})}.\qedhere $$
\end{proof}

\subsection{Topological and reduced ideal zeta functions}\label{subsec:top.red}
The next corollaries concern the \emph{topological} and \emph{reduced
  ideal zeta functions} associated to the Lie rings
$L_{m,n}$. Informally, these are two related (but distinct) limiting
objects capturing the behaviour of $\zeta^{\nl}_{L_{m,n}(\lri)}(s)$ as
`$q\rarr 1$'; see \cite{Rossmann/15} and \cite{Evseev/09},
respectively, for details and precise definitions.

For our purposes, the following \emph{ad hoc} definitions may
suffice. Let $Z(s) = I_n(q^{-1};\left(x_i\right)_{i=1}^n)$ for
numerical data $x_i = q^{a_i-b_is}$, for integers $a_i\in\N_0$,
$b_i\in\N$. Define the \emph{topological zeta function}
$Z_{\topo}(s)\in\Q(s)$ via
$$Z(s) = Z_{\topo}(s) (q-1)^{-n} + O((q-1)^{-n+1})$$ and the
\emph{reduced zeta function}
$$Z_{\red}(Y) := I_n(1;(Y^{b_i})_{i=1}^n)\in\Q(Y).$$ We omit the
proofs of the following simple calculations.
\begin{lemma}\
\begin{enumerate}
 \item $Z_{\topo}(s) = \frac{n!}{\prod_{i=1}^n(b_is-a_i)},$
 \item
   $Z_{\red}(Y) = \frac{\sum_{w\in S_n} \prod_{j\in \Des(w)}
   Y^{b_j}}{\prod_{i=1}^n(1-Y^{b_i})}.$ 
\end{enumerate}
\end{lemma}

\begin{corollary}
$$\left.Z_{\red}(Y)(1-Y)^n\right|_{Y=1} =  s^{-n} \left. Z_{\topo}(s^{-1})\right|_{s=0} = \frac{n!}{\prod_{i=1}^n b_i} \in \Q_{>0}.$$
\end{corollary}

\begin{corollary}\label{cor:topo}\
  \be
\item The topological ideal zeta function of $L_{m,n}$ is given by
$$\zeta^{\nl}_{L_{m,n},\topo}(s) =
\frac{n!}{\left(\prod_{j=0}^{d(m,n)-1}(s-j)\right)\left(
    \prod_{i=0}^{n-1}(b_i^{\nl}(m,n)s-
    a_i^{\nl}(m,n))\right)}\in\Q(s).$$
It has degree $-h(m,n)$ in $s$, a simple pole at $s=0$ with residue
$\frac{(-1)^{h(m,n)-1}}{(h(m,n)-1)!}$ and satisfies
$$\left.s^{-h(m,n)} \zeta^{\nl}_{L_{m,n},\topo}(s^{-1})\right|_{s=0} =
  \frac{n!}{\prod_{i=0}^{n-1} b_i^{\nl}(m,n)} =:
  \mu^{\nl}_{m,n}\in\Q_{>0},$$ a nonzero rational number satisfying
  $\mu^{\nl}_{m,n} h(m,n)!\in\N$.
\item The reduced ideal zeta function of $L_{m,n}$ is given by
  $$\zeta^{\nl}_{L_{m,n},\textup{red}}(Y) = \frac{\sum_{w\in S_n} \prod_{j\in \Des(w)} Y^{{b^{\nl}_{n-j}(m,n)}}}{(1-Y)^{d(m,n)} \prod_{i=0}^{n-1} (1-Y^{b^{\nl}_{i}(m,n)})}.$$
  It has degree $-d(m,n)-h(m,n)$ in $Y$, a pole of order $h(m,n)$ at
  $Y=1$, and satisfies
  $$
  \left.\zeta^{\nl}_{L_{m,n},\textup{red}}(Y)(1-Y)^{h(m,n)}\right|_{Y=1}
  = \mu^{\nl}_{m,n} \in \Q_{>0}.$$\ee
\end{corollary}

\begin{example}
  For $(m,n)=(2,3)$ (see Example~\ref{exa:23}) we obtain
  $$\zeta^{\nl}_{L_{2,3},\textup{top}}(s) = \frac{1}{5\left( \prod_{i=0}^8 (s-i)\right) (4s-9)(s-2)(7s-11)}$$
  and
  $$ \zeta^{\nl}_{L_{2,3},\textup{red}}(Y) = \frac{1 + 2 Y^7 + 2 Y^{10} + Y^{17}}{(1-Y)^9 (1-Y^7)(1-Y^{10})(1-Y^{12})},$$
  whence
  $$\mu_{2,3}^{\nl} = \frac{1}{140}.$$
\end{example}

\begin{remark}
  Together, corollaries~\ref{cor:pad} and \ref{cor:topo} confirm the
  conjectures in \cite[Sections~8.1 and~8.2]{Rossmann/15} in the
  relevant special cases.

  It remains an interesting challenge to give an intrinsic, algebraic
  interpretation of the ``multiplicities'' $\mu_{m,n}^{\nl}$. That
  they occur as invariants of both the reduced and the topological
  zeta functions seems remarkable.
\end{remark}

\subsection{Graded ideal zeta functions}
Let $R$ be a ring as in Section~\ref{subsec:ideal}. The \emph{graded
  Lie algebra associated to} $L_{m,n}(R) := L_{m,n}\otimes_\Z R$ is the
$R$-Lie algebra
$$\gr L_{m,n}(R) = \underbrace{L_{m,n}(R) / L_{m,n}(R)'}_{=:L^{(1)}}
\oplus \underbrace{L_{m,n}'(R)}_{=:L^{(2)}}.$$ An $R$-ideal $I$ of
$\gr L_{m,n}(R)$ is \emph{graded} if
$I = (I\cap I^{(1)}) \oplus (I\cap I^{(2)})$. The \emph{graded ideal
  zeta function} of $L_{m,n}(R)$ is the Dirichlet series
$$\zeta^{\nl_{\gr}}_{L_{m,n}(R)}(s) = \sum_{I \nl_{\gr} \gr
  L_{m,n}(R)} | \gr L_{m,n}(R):I|^{-s};$$ enumerating the graded
ideals of $\gr L_{m,n}(s)$ of finite index; cf.\ \cite{Rossmann/18}
and \cite{LeeVoll/18}. One advantage of writing
$\zeta^{\nl}_{L_{m,n}(\lri)}$ in terms of the generating function
$A^{\nl}_{m,n}(q,q^{-s})$ defined in \eqref{equ:Amn} (see
\eqref{equ:zeta.AJM}) is that a trivial modification yields a formula
for the graded ideal zeta function. Indeed,
 \begin{equation*}\label{equ:graded.zeta.AJM}
   \zeta^{\nl_{\gr}}_{L_{m,n}(\lri)}(s) =\zeta_{\lri^{d(m,n)}}(s)
   \frac{1}{1-q^{-s\,h(m,n)}}A^{\nl_{\gr}}_{m,n}(q,q^{-s}),
 \end{equation*}
 where
 $$A^{\nl_{\gr}}_{m,n}(q,q^{-s}) = \sum_{[\Lambda'] \in \mcV_n}q^{-s
   w'([\Lambda'])};$$ cf.\ \cite[Example~1.6]{LeeVoll/18}. Modifying
 the computation in Section~\ref{subsec:completion} yields the
 following result.

\begin{theorem}\label{thm:main.gr}
$$\zeta^{\nlgr}_{L_{m,n}(\lri)}(s) = \zeta_{\lri^{d(m,n)}}(s) \cdot
I_n\left(q^{-1};\left( q^{i(n-i) -s b^{\nl}_i(m,n)}
  \right)_{i=n-1}^0\right),$$
where, for $i\in\{0,1,\dots,n-1\}$, the numerical data
$b^{\nl}_i(m,n)$ is as in \eqref{num.data} in Theorem~\ref{thm:main}.
\end{theorem}

All the results recorded in Section~\ref{subsec:global} to
\ref{subsec:top.red} have ``graded analogues''. We only note here the
behaviour of $\zeta^{\nlgr}_{L_{m,n}(\lri)}(s)$ at $s=0$, to be
compared with Corollary~\ref{cor:pad}.

\begin{corollary}
$$\left.\frac{\zeta^{\nlgr}_{L_{m,n}(\lri)}(s)}{\zeta_{\lri^{d(m,n)}}(s)\zeta_{\lri^n}(s)}\right|_{s=0}
  = \frac{n}{h(m,n)}.$$
\end{corollary}

This behaviour is analogous to that observed for some (and conjectured
for all) free nilpotent Lie rings, but not universal;
cf.\ \cite[Conjecture~6.11 and Remark~6.13]{LeeVoll/18}.

\subsection{Representation zeta functions}

Let $\bfG_{m,n} = \bfG_{L_{m,n}}$ be the unipotent group scheme
associated to the nilpotent Lie ring $L_{m,n}$ as in
\cite[Section~2.4]{StasinskiVoll/14}. Given a ring of integers $\Gri$
of a number field~$K$, the group $G = \bfG_{m,n}(\Gri)$ is a finitely
generated torsion-free nilpotent group of nilpotency class $2$ and
Hirsch length $h(m,n)\cdot |K:\Q|$. (For $\Gri=\Z$ we recover the
groups $\Delta_{m,n}=\bfG_{m,n}(\Z)$ from \cite{BermanKlopschOnn/18}.)
Denote by
$$\zeta_{G}(s) = \sum_{n=1}^\infty \wt{r}_n(G)n^{-s}$$ the
\emph{representation zeta function} of $G$, encoding the numbers
$\wt{r}_n(G)$ of twist-isoclasses of irreducible complex
$n$-dimensional representations of $G$; see, for instance,
\cite[Section~1.1]{StasinskiVoll/14}, for background. 

\begin{theorem}\label{thm:rep} For all $m,n\in\N$,
$$\zeta_{\bfG_{m,n}(\Gri)}(s) = \frac{\zeta_K(se(m,n)-n)}{\zeta_K(se(m,n))}.$$
\end{theorem}

\begin{proof} By \cite[(1.4)]{StasinskiVoll/14}, the representation
  zeta function $\zeta_{\bfG_{m,n}(\Gri)}(s)$ is an Euler product of
  representation zeta functions of the form
  $\zeta_{\bfG_{m,n}(\Gri_\mfp)}(s)$, where $\Gri_{\mfp}$ is the
  completion of $\Gri$ at a nonzero prime ideal $\mfp$. We fix such
  an ideal $\mfp$ and write $\lri = \Gri_{\mfp}$ and $q=q_{\mfp}$ for
  the residue field cardinality of $\lri$. We use the notation and
  results of \cite[Section~2]{StasinskiVoll/14}, specifically those
  for nilpotency class $2$ in \cite[Section~2.4]{StasinskiVoll/14}, to
  compute the rational function
  $\zeta_{\bfG_{m,n}(\lri)}(s)\in\Q(q^{-s})$ (cf.\
  \cite[Corollary~2.19]{StasinskiVoll/14}) explicitly. The quantity
  $r$ is equal to $d(m,n)$. For $N>0$ we find
  $W_N(\lri) = (\lri/\mfp^N)^n \setminus (\mfp/\mfp^N)^n$, whereas
  $W_0(\lri) = \bf0$, whence
$$\# W_N(\lri) = \begin{cases} (1-q^{-n})q^{nN} & \textup{ if } N>0,\\ 1 & \textup{ otherwise.}\end{cases}$$
One checks immediately that, for all $N\in\N_0$,
$$\mcN^{\lri}_{N,\bfa} = \# W_N(\lri) \, \delta_{\bfa = ((0)^{(e(m,n))},(N)^{(\lfloor d(m,n)/2 \rfloor -e(m,n))})}.$$
Indeed, by Lemma~\ref{lem:coma}, for $N>0$ and any
$\bfy\in W_N(\lri)$, the two matrices
$$M_{m,n}(\bfy) \text{ and } \left( \begin{matrix} & -\Id_{e(m,n)} & \\ \Id_{e(m,n)} && \\&& \end{matrix}\right)\in\Mat_{d(m,n)}(\lri/\mfp^N).$$
are equivalent. By \cite[Proposition~2.18]{StasinskiVoll/14}, it
follows that
\begin{align*}
  \zeta_{\bfG_{m,n}(\lri)}(s) &= \sum_{N\in\N_0,\, \bfa \in\N_0^{\lfloor d(m,n)/2\rfloor}} \mcN^{\lri}_{N,\bfa} q^{-s\sum_{i=1}^{\lfloor d(m,n)/2 \rfloor} (N-a_i)s} \\
                              &= 1 + \sum_{N\in\N} (1-q^{-n})q^{nN}(q^{-s})^{Ne(m,n)}\\ &= 1 + (1-q^{-n})\frac{q^{n-se(m,n)}}{1-q^{n-se(m,n)}} = \frac{1-q^{-se(m,n)}}{1-q^{n-se(m,n)}}.
\end{align*}
The result follows from the well-known Euler factorization
$\zeta_K(s) = \prod_{\mfp} (1-q_{\mfp}^{-s})^{-1}$.
\end{proof}

\begin{remark}
For $n=1$, Theorem~\ref{thm:rep} generalizes the well-known formulae
for the representation zeta functions of the Heisenberg groups
$\bfH(\lri) = \bfG_{m,1}(\lri)$; see Example~\ref{exa:heisenberg} and
\cite[Theorem~B]{StasinskiVoll/14}. For $m=1$, we recover the
representation zeta functions of the Grenham groups $G_{n+1}$ (see
Example~\ref{exa:grenham}), computed by Snocken in his PhD-thesis
(see~\cite[Example~6.2]{Snocken/14}).
\end{remark}

We note an immediate consequence of Theorem~\ref{thm:rep} regarding
the topological representation zeta function
$\zeta_{\bfG_{m,n},\topo}(s)\in\Q(s)$ of $\bfG_{m,n}$;
cf.\ \cite[Definition~3.6]{Rossmann/16b}.

\begin{corollary}
$$\zeta_{\bfG_{m,n},\topo}(s) = \frac{se(m,n)}{se(m,n)-n}.$$
\end{corollary}

Consequently, all questions raised in
\cite[Section~7]{Rossmann/16b}---except possibly Question~7.3---have
positive answers for the group schemes $\bfG_{m,n}$.

\begin{funding}
  I am grateful for support by the German-Israeli Foundation for
  Scientific Research and Development (GIF) through grant no.\ 1246. 
\end{funding}

\begin{acknowledgements}
  I thank Angela Carnevale for stimulating discussions and a very
  helpful $\textsf{maple}$-sheet. To Tobias Rossmann I am grateful for
  many valuable conversations, in particular about topological and
  reduced zeta functions. I thank Uri Onn for teaching me many things
  over the years (including the word ``porism'') and Mark Berman for
  helpful comments on a draft of this paper. To an anonymous referee I
  am greatly indebted for numerous insightful comments and corrections
  of errors and oversights.
\end{acknowledgements}

%\bibliographystyle{abbrv}%{amsplain}
%\bibliography{masterbibliography_may2014}
\def\cprime{$'$}

\end{document}